\documentclass{amsart}
\usepackage{amssymb}
\usepackage{amsfonts}

\newtheorem{theorem}{Theorem}
\theoremstyle{plain}

\newtheorem{corollary}{Corollary}

\newtheorem{definition}{Definition}
\newtheorem{example}{Example}

\newtheorem{proposition}{Proposition}

\numberwithin{equation}{section}

\begin{document}
\title[Homotopy]{Homotopy groups and twisted homology of arrangements}
\author{Richard Randell}
\address{Department of Mathematics,
University of Iowa, Iowa City, IA 52242 USA}
\email{randell@math.uiowa.edu}
\date{May 5, 2009}

\begin{abstract}
Recent work of M. Yoshinaga \cite{Yoshinaga} shows that in some instances certain higher homotopy groups of arrangements map onto non-resonant homology. This is in contrast to the usual Hurewicz map to untwisted homology, which is always the zero homomorphism in degree greater than one. In this work we examine this dichotomy, generalizing both results.  
\end{abstract}

\maketitle

\section{Introduction}

Let $\mathcal{A}$ be an arrangement of hyperplanes in $\mathbf{C}
^{\ell }$. Thus $\mathcal{A}$ is a finite non-empty collection $
\left\{ H_{1},\ldots ,H_{n}\right\} $ where $H_{i}=\alpha _{i}^{-1}(b_i)$ with $b_i \in \mathbf{C}$ and
each $\alpha _{i}$ is a linear homogeneous form in the variables $
(z_{1},\ldots ,z_{\ell })$. (See \cite{OrlikTerao} for material on arrangements). \
We call $\mathcal{A}$ an $\ell $\emph{-arrangement}. We let $M$ be the
complement of the union $H$ of the hyperplanes 
\begin{equation*}
M=\mathbf{C}^{\ell }\setminus \cup H_{i}.
\end{equation*}

Now $M$ is the complement of a real codimension two subset of $\mathbf{R}
^{2\ell }$ and so has numerous topological properties of interest.  Here we
focus on the homotopy theory of the complement.  Since the fundamental
group of the complement $\pi=\pi _{1}(M)$is fairly rich in structure, it makes
sense to look at covers of the complement.  Equivalently one may look at
homology or homotopy with coefficients in $R=\mathbb{Z}[\pi]$ modules or
at local system homology.  Two results motivate this study.  The
first is found in \cite{RandellHomotopy}; the untwisted Hurewicz
homomorphism $h:\pi _{k}(M)\rightarrow~H_{k}(M;\mathbb{Z})$ is trivial for $
k>1$. \ On the other hand, in \cite{Yoshinaga} Yoshinaga has shown that the
twisted Hurewicz homomorphism maps \emph{onto} certain twisted homology for
non-resonant local systems, in the case that $M$ is a generic section of
another hyperplane arrangement.  In this work we generalize both these results
while giving a unified treatment.  

It is known that an arrangement complement has a minimal CW structure (\cite{DimcaPapadima}, \cite{RandellMinimal}), and this fact is used in Yoshinaga's proof.  Part of our goal in this work is to understand exactly where this minimality is useful in studying the topology of arrangement complements.  Thus we assume minimality only when necessary, and in particular reprove Yoshinaga's result without use of this property.  

We begin in the next section with a discussion of the twisted homology and Hurewicz homomorphisms and general relationships between higher homotopy groups and twisted homology groups.  We then specialize to the case of hyperplane arrangement complements, and derive consequences in various situations.    

\section{Local Systems and the Twisted Hurewicz homomorphism}

Let $(X,Y)$ be an $(n-1)-$connected topological pair with $n\geq 3$.  Thus $\pi _{j}(X,Y)\cong 0$ for $j\leq n-1$, and by the long exact
homotopy sequence of the pair, $\pi _{j}(X)\cong \pi _{j}(Y)$ for $j\leq n-2$.  In particular, the fundamental group $\pi =\pi _{1}(Y)$ of $Y$ includes
isomorphically into the fundamental group of $X$.  Then there is a
generalized (twisted) Hurewicz isomorphism, which we now explain.  We
follow the discussion of Rong \cite{Rong}, with notation as in Hatcher \cite{Hatcher}.  Let $N$ be a left $\mathbf{Z}[\pi ]$-module,
characterized by the left action

\begin{equation*}  
\varphi : \pi_{1}(X) \times N \rightarrow N
\end{equation*}

or equivalently by the homomorphism 

\begin{equation*}  
\varphi : \pi_{1}(X) \rightarrow Aut(N).
\end{equation*}
 Then as usual there is an associated local system $N_{\varphi}$ on $X$.
 (See Hatcher \cite{Hatcher} for details about local systems).  Then one
has homology and cohomology with coefficients in $N_{\varphi }$, and by the
usual left action of $\pi $ on $\pi _{n}(X,Y)$ we may form the tensor product
\begin{equation*}
\pi _{n}(X,Y)\otimes _{\pi }N
\end{equation*}
where we use the subscript``$\pi "$ to indicate the
tensor product over $\mathbf{Z}[\pi ]$.  Note that since we are working over a generally noncommutative ring, we need to make the action on $\pi_n(X,Y)$ into a \emph{right} action, as in \cite{Hatcher}.   In what follows we will switch as appropriate among the various viewpoints and notations for local systems, but we will always have the above set-up in mind.  We will generally suppress $\varphi$ from the notation.

\begin{theorem}
(Twisted Hurewicz Theorem) Let $(X,Y)$ be an $(n-1)-$connected topological pair with $n\geq 3$.  Then there is a natural isomorphism
\begin{equation}
h:\pi_n(X,Y) \otimes _{\pi } N\rightarrow H_{n}(X,Y;N)
\end{equation}
\end{theorem}

Notice that the right hand side involves local system homology, while the
left hand side is an algebraic tensor product.  We sketch the proof of this theorem below to highlight the property of naturality.  Here ``Naturality" means that with
these assumptions there exist homomorphisms $h:\pi_n(X)\otimes _{\pi } N \rightarrow H_{n}(X;N) $ and $h:\pi_n(Y) \otimes _{\pi } N\rightarrow H_{n}(Y;N)$ so that the following diagram
commutes.
\begin{equation}
\begin{array}{ccccccc}
\rightarrow  & H_{n}(X;N) & \rightarrow  & H_{n}(X,Y;N)
& \stackrel{i_*}{\rightarrow}  & H_{n-1}(Y;N) & \rightarrow  \\ 
& \uparrow h &  & \uparrow h\cong  &  & \uparrow h &  \\ 
\rightarrow  & \pi_n(X) \otimes _{\pi } N & \rightarrow  & \pi_n(X,Y) \otimes _{\pi } N & \rightarrow  & \pi_{n-1}(Y) \otimes _{\pi } N & \rightarrow 
\end{array}
\end{equation}
\begin{equation*}
\begin{array}{cccc}
\rightarrow H_{n-1}(X;N) & \rightarrow  & H_{n-1}(X,Y;N) & \rightarrow  \\ 
\uparrow h &  & \uparrow h &  \\ 
\rightarrow \pi _{n-1}(X) \otimes _{\pi } N & \rightarrow  &\pi _{n-1}(X,Y) \otimes _{\pi
}N & \rightarrow 
\end{array}
\end{equation*}
\pagebreak
\begin{proof}(Sketch \cite{Rong})  Let $(\tilde{X},\tilde{Y})$ denote the universal covers.  There is a sequence of isomorphisms 

$H_n((X,Y);N)$ 

$\cong H_0(\pi,H_n(\tilde{X},\tilde{Y};N))$, by the Eilenberg-Moore spectral sequence

$\cong [H_n(\tilde{X},\tilde{Y};N)]_\pi$

$\cong H_n(\tilde{X},\tilde{Y})\otimes_{\pi}N$, by the universal coefficient theorem

$\cong \pi_n(X,Y)\otimes_{\pi}N$, by the usual Hurewicz theorem.

Naturality may be traced through these isomorphisms.

\end{proof}

Note that since tensor product is not exact, we have no guarantee that the
lower row of this diagram is exact.  It will be exact, of course, if the module $N$ is flat.  Of course the upper row is exact always. 
Then we have the following general result:

\begin{theorem} Let $(X,Y)$ be an $(n-1)-$connected topological pair with $n\geq 3$.
Then $\ker (i_{\ast }:H_{n-1}(Y;N)\rightarrow H_{n-1}(X;N))\subset im(h_{Y}:N\otimes _{\pi }\pi _{n-1}(Y)\rightarrow
H_{n-1}(Y;N))$.
\end{theorem}

\begin{proof}
Since $\pi _{n-1}(X,Y)\cong 0$, and tensor product is right exact, the
sequence 
\begin{displaymath}
\begin{array}{ccccccc}
\pi _{n}(X,Y) \otimes_{\pi } N &  \rightarrow & \pi_{n-1}(Y)\otimes _{\pi }N & \rightarrow & \pi _{n-1}(X)\otimes _{\pi }N & \rightarrow & 0
\end{array}
\end{displaymath}
 is exact,
so that the commuting diagram above yields a commuting diagram with exact
rows
\begin{equation}
\begin{array}{cccccccc}
H_{n}(X,Y;N) & \rightarrow  & H_{n-1}(Y;N) & \stackrel{i_*}{\rightarrow}  & H_{n-1}(X;N) & \rightarrow  & 0 & \rightarrow  \\ 
\uparrow h\cong  &  & \uparrow h_{Y} &  & \uparrow h &  &  &  \\ 
\pi _{n}(X,Y) \otimes _{\pi }N & \rightarrow  &\pi _{n-1}(Y) \otimes _{\pi } N
& \rightarrow  & \pi_{n-1}(X) \otimes_{\pi } N & \rightarrow  & 0 & 
\rightarrow 
\end{array} \label{eq:BasicDiagram}
\end{equation}

A simple diagram chase then gives the result.
\end{proof}

Next we examine the above result in various special cases.

\section{Consequences of the Twisted Hurewicz Theorem}

The following result follows immediately from the exact sequence (\ref{eq:BasicDiagram}).

\begin{proposition}Let $(X,Y)$ be an $(n-1)-$connected topological pair with $n\geq 3$.
If $h_{Y}$ is the zero homomorphism, then $i_{\ast }:H_{n-1}(Y;N)\rightarrow H_{n-1}(X;N)$ is injective and hence an isomorphism. 

\end{proposition}

\begin{definition}

An \emph{arrangement pair} $(M,M'')$ is an $(n-1)$-connected pair of topological spaces with $M$ an arrangement complement and $M''$ a generic hyperplane section of $M$.
\end{definition}
The notation is usual; $M''$ is the ``restriction" of $M$ to a hyperplane.  A hyperplane section is generic provided that the intersection lattice of $M''$ agrees with that of $M$ through rank $n-1$.

It is well-known that a generic section $M''$ of $M$ yields an arrangement pair with the homotopy type of a CW pair, and that in fact the number of $n$-cells attached to $M''$ to yield $M$ is exactly equal to the $n$-th betti number of $M$.  (See (\cite{DimcaPapadima}, \cite {RandellMinimal}).  

\begin{proposition}
If $N$ is the trivial system, $N=\mathbb{Z}$, and $(M,M'')$ is an
arrangement pair, then $i_{\ast }$ is an isomorphism, as is $j_{\ast
}:H_{n}(M;N)\rightarrow H_{n}(M,M'';N)$.  
\end{proposition}
\begin{proof}For the trivial system, it was shown in 
\cite{RandellHomotopy} that the Hurewicz map is always trivial on higher homotopy groups.  Thus $h_{M''}=0$.  It then follows from (\ref{eq:BasicDiagram}) that $i_{\ast}$ and $j_{\ast}$ are isomorphisms.
\end{proof}

In the case of hyperplane arrangments, the injectivity of $i_{\ast }$
follows from the Orlik-Solomon algebra (or, more basically from the Lefschetz theorem on
hyperplane sections).  More interesting is the next result, as mentioned
due to Yoshinaga in the case of a generic pair of hyperplane complements and a
non-resonant local system $N$, where $N$ as an abelian group is just $\mathbf{C}^r$ (a local system of rank $r$).  Our method of proof uses nothing about the
\emph{minimal} CW structure of arrangement complements, however.

\begin{corollary}
Suppose $(X,Y)$ is an $(n-1)-$connected pair of topological spaces, $n\geq 3$
.  Suppose that $N$ is a local system on $X$ so that $
H_{n-1}(X;N)=0$. \ Then $h_{Y}:\pi_{n-1}(Y) \otimes _{\pi } N \rightarrow H_{n-1}(Y;N)$ is onto.
\end{corollary}

Stated for arrangement pairs we have

\begin{theorem}(Yoshinaga \cite{Yoshinaga})
Suppose $(M,M'')$ is an arrangement pair and $N$ is a non-resonant local system of rank $r$ on $M$.  Then 
$h_{M''}:\pi
_{n-1}(M'')\otimes _{\pi } N \rightarrow H_{n-1}(M'';N)$ is onto.
\end{theorem}
\begin{proof}
For a non-resonant system $
H_{n-1}(M;N) = 0$.
\end{proof}

Actually the Hurewicz map here differs slightly from the one considered by Yoshinaga, but the result above clearly implies that of Yoshinaga.  Here is the relationship.  Let $X$ be a topological space with basepoint $x$.  Then, as in \cite{Yoshinaga} we have a twisted Hurewicz map

\[h_j:\pi_j(X,x) \otimes_{\mathbf{Z}}\mathcal{L}_x \rightarrow H_j(X,\mathcal{L})\] 
defined by setting $h(f \otimes t)$ equal to the twisted cycle it determines.  This Hurewicz homomorphism differs by a change of ring (from $\mathbf{Z}[\pi]$ to $\mathbf{Z}[1] \cong \mathbf{Z}$) from the one we considered earlier.  The homomorphisms are related by the obvious commuting triangle.

\section{The Image of Hurewicz Maps}
\subsection{Basic Topological Properties of Arrangement Complements}

An arrangement complement $M$ has several basic topological properties.  Before listing them we need a bit of notation \cite{OrlikTerao}.  To any hyperplane arrangement one has the lattice whose elements are the various intersections of the hyperplanes, ordered by reverse inclusion.  To any lattice element $Z$ one has the ``localization" $\mathcal{A}_Z=\{H \in \mathcal{A}  \mid Z \subset H\}$.  The rank $r=rk(Z)$ of any lattice element $Z$ is its codimension.  Note that it is always the case that there is the inclusion $M(\mathcal{A}) \subset M(\mathcal{A}_Z)$.  

Here are several basic properties of arrangement complements.

\begin{description}
\item[Locality] The ``Brieskorn homomorphism" $i:\oplus H^r(M(\mathcal{A}_Z)) \rightarrow H^r(M(\mathcal{A}))$ is an isomorphism, the direct sum taken over all rank $r$ lattice elements.

\item[Toroidality] The cohomology of $M$ is generated by cohomology in degree one, in particular by the logarithmic one forms $d\alpha_i / \alpha_i$, where $\alpha_i$ is a linear form defining the hyperplane $H_i$.

\item[Minimality]  $M$ has the homotopy type of a minimal CW-complex, where ``minimal" here means that the number of $q$ cells is equal to the $q$-th betti number, for all $q$.

\end{description}

In the context of local systems locality is not usually relevant, since the Brieskorn homomorphisms are usually not defined: to do so one needs a local system on $M(\mathcal{A})$ which is the restriction of one on  $M(\mathcal{A}_Z)$, and so is trivial around hyperplanes not containing $Z$.  Deletion-restriction for this situation (when $rk(Z)=1$) is examined by D. Cohen in \cite{CohenDR}.  

Minimality in the local systems setting has been examined in \cite{DimcaPapadima2}.  In particular they show there that the $\pi$-equivariant chain complex associated to a Morse-theoretic minimal CW structure on an arrangement complement is independent of the CW structure.

In this section we will examine the role of toroidality for local systems.  We will prove the analog of the result mentioned earlier, that the Hurewicz map is trivial for the trivial local system, finding a part of the kernel for any local system.  Later we will speculate about injectivity--can the Hurewicz map in local systems detect homotopy groups?

Now for a non-resonant (or ``generic") local system it is known that the first homology and cohomology groups vanish (\cite{CDO}).  Thus for generic local systems, toroidality fails badly.  Toroidality is a key ingredient in the proof that the usual Hurewicz homomorphism is zero.  In place of that result we have the result below, which first requires some preparation.

\subsection{The Hurewicz Image for General Arrangement Covers}

Now since $M$ is an arrangement complement, it is path-connected, locally path-connected and semi-locally simply connected.  Therefore $M$ has a universal cover and there is a bijective correspondence between subgroups $\pi'$ of $\pi = \pi_1(M)$ and connected covering spaces of $M$.  Let $M'$ be the cover corresponding to $\pi'$.  Then the free abelian group $\mathbf{Z}[\pi/\pi']$ with basis the cosets $\gamma \pi'$ is a  $\mathbf{Z} [\pi]$-module and the homology groups of $M'$ with coefficients in $\mathbf{Z}$ are the same as the homology groups of $M$ with coefficients in the local system $\mathbf{Z}[\pi/\pi']$.  The same holds in cohomology and there are isomorphisms 

\[ H_j(M,\mathbf{Z}[\pi/\pi']) \cong H_j(M',\mathbf{Z}) \] 
and
\[ H^j(M,\mathbf{Z}[\pi/\pi']) \cong H^j(M',\mathbf{Z}) \]

Let us now consider the case where we have a local system  on $M$ of the form $N=\mathbf{Z}[\pi/\pi'])$ .

\begin{proposition}$Im(h': \pi_j(M') \rightarrow H_j(M')) \subseteq ker (p_*: H_j(M') \rightarrow H_j(M))$.

\end{proposition}

Note that $H_j(M') = H_j(M,\mathbf{Z}[\pi/\pi'])$.  In case $M=M'$ this is the result of \cite{RandellHomotopy}.
\begin{proof}  Fix $j \geq 2$.  Let $p:\tilde{M} \rightarrow M$ be the universal cover of $M$.  Then there is a commuting diagram
\begin{equation}
\begin{array}{ccc} \pi_j(\tilde{M}) & \stackrel{\tilde{h}}{\rightarrow}  & H_{j}(\tilde{M}) \\

p_* \downarrow \cong & \ & \downarrow p_*  \\

\pi_j(M') & \stackrel{h'}{\rightarrow} & H_j(M') \\

p_* \downarrow \cong & \ & \downarrow p_*  \\

\pi_j(M) & \stackrel{h}{\rightarrow} & H_j(M)  \\
 
\end{array}
\end{equation}

Now by the result of \cite{RandellHomotopy} the bottom horizontal arrow is the zero homomorphism, and the result follows
\end{proof}

Note that one can tensor the left-hand higher homotopy groups with $\mathbf{Z}[\pi])$-modules, maintaining the commutativity and the isomorphisms in the left-hand column, to get analogous results.
We further note that finite covers of $M$, such as the Milnor fiber, fit into this framework. For instance, in the case of the Milnor fiber one may see easily that one does not have isomorphism in the above proposition.

Examples below show that the inclusion in the above Proposition may be strict.
\subsection{Examples}
\begin{example}

Consider the reflection arrangement $A_3$, as an arrangement in projective space.  In that case $H_1(M)$ is free abelian of rank five, while the six-fold cyclic cover $M'$ (which is the Milnor fiber $F$ of the associated central arrangement) has $H_1(F)$ free of rank seven.  One may use standard calculations and euler characteristic arguments to conclude that the second homology groups have ranks six and eighteen respectively.  Also, $M$ and $M'=F$ are aspherical, so that $Im(h':~\pi_j(M') \rightarrow H_j(M'))$ is zero, while it may be seen easily that $p_*: H_2(M')~\rightarrow H_2(M)$ has non-trivial kernel (in fact, $p_*$ is surjective, so that the kernel is free abelian of rank twelve.)

\end{example}

The next example shows that the Hurewicz map on the universal cover may be an isomorphism, and that this is reflected in twisted homology.
\begin{example}

 Let $\mathcal{A}$ be any arrangement, with complement $M$, and let $\tilde{M}$ be its universal cover.  Then on the first non-trivial homotopy group, the Hurewicz map is an isomorphism.  For example, one may take $\mathcal{A}$ to be the arrangement of the three hyperplanes $x=0$, $y=0$, $x+y-1=0$ in $\mathbf{C^2}$.  Now by a well-known result of A. Hattori \cite{Hattori}, the complement $M$ has the homotopy type of the two-skeleton of the three torus $T^3$.  That is, $M$ has the homotopy type of $T^3$ with a single $3$-cell removed.  It is then easily seen that $F$ has the homotopy type of $T^3$ with four $3$-cells removed.  Then $\pi_2(M)$ is the free $Z[\pi]$-module of rank one (where $\pi = \pi_1(M) $ is the free abelian group of rank three.)  Then the Hurewicz homomorphism

\begin{displaymath}
\begin{array}{ccccc}
h:\pi_2(M) \otimes_\pi Z[\pi] & \rightarrow & H_2(\tilde{M}) & \cong & H_2(M,\mathbf{Z}[\pi])
\end{array}
\end{displaymath}
is an isomorphism. 
\end{example}

In fact, one recalls the following well-known fact (Hurewicz) on \emph{detectability} of higher homotopy groups.  We will say (an element of) a homotopy group is detectable if there is some twisted Hurewicz map $h$ which does not send it to zero.

\begin{theorem}The first non-trivial higher homotopy group of an space $X$ possessing a universal cover $\tilde{X}$ is detected by local system homology.
\end{theorem}
\begin{proof}The groups $\pi_k(X)$ and  $\pi_k(\tilde{X})$ are isomorphic, and the Hurewicz theorem gives an isomorphism between $\pi_k(\tilde{X})$,  and $H_k(\tilde{X})$. The latter is isomorphic to $H_k(X;\mathbf{Z}[\pi])$.
\end{proof}

The next example points out that the twisted Hurewicz map may have the maximal image, subject to the above results. 

\begin{example}\label{generic example}
Take $\mathcal{A}$ to be $x=0$, $y=0$, $x+y-1=0$ in $\mathbf{C^2}$ giving $M$ as in the previous example.  Consider the Milnor fiber $F$ associated to the cone of this arrangement, $\{x=0, y=0, x+y-z=0, z=0\}$. This fiber {F} is the four-fold cyclic cover of $M$ associated to the homomorphism sending each meridianal loop of $M$ to the generator of $\mathbf{Z}/4\mathbf{Z}$.    Since $F$ is a cover of $M$,  $\pi_2(F)$ is also the free $Z[\pi]$ module of rank $1$, where $\pi$ is free abelian of rank $3$, and of course the homomorphism induced by the covering map is an isomorphism.

Now regarding $\pi_2(F)$ as a module over $\mathbf{Z}[\pi_1(F)]$ one can show that the image of $\pi_2(F) \otimes_\pi \mathbf{Z}$ in $H_2(F)$ is free of rank three.  Thus this image is of finite index in $ker(p_*)$, which is easily seen to be of rank three.

\end{example}
Finally, we may consider the same arrangement, but with various rank one local systems.
\begin{example}Take the same arrangement as in Example \ref{generic example}, but take rank one local systems $\mathcal{L}_a$ corresponding to the automorphisms sending all generators of first homology to $a=1$, $a=i$, $a=-1$, $a=-i$ respectively.  Then by Corollary 1.5 of \cite{cohensuciu} the homology $H_2(F;\mathbf{C})$ is the direct sum of the four corresponding local system homology groups $H_2(F;\mathcal{L}_a)$ where $a=1,i,-1,-i$.  The latter three values of $a$ yield ``non-resonant"' local systems, with homology concentrated in degree two (of dimension one).  The value $a=1$ yields the homology of $M$ with betti numbers $(1,3,3)$.  In each case $a\neq1$ the twisted Hurewicz map \[h_2:\pi_2(M,m) \otimes_{\mathbf{Z}}\mathcal{L}_m \rightarrow H_2(M,\mathcal{L}_a)\] is onto. 
\end{example}

We close with a general

\textbf{Question:}  Let $\mathcal{A}$ be any complex hyperplane arrangement in $\mathbf{C}^{\ell}$, $\rho \in \pi_k(M)$ with $k \leq \ell$.  When is there a rank one local system $\mathcal{L}$ on $M$ so that for the appropriate Hurewicz map $h:\pi_k(M) \otimes_Z \mathcal{L}_m \rightarrow H_k(M;\mathcal{L})$, one has $h(\rho) \neq 0$?


\begin{thebibliography}{99}
\bibitem{CohenDR}Cohen, Daniel C., Triples of arrangements and local systems, Proc. Amer. Math. Soc. \textbf{130}, (2002), no. 10, 3025-3031.
\bibitem{CDO} Cohen, D.C., Dimca, A., and Orlik, P., Nonresonance conditions for arrangements, Ann. Inst. Fourier (Grenoble) \textbf{53}, (2003), no. 6, 1883-1896.
\bibitem{cohensuciu}Cohen, Daniel C. and Suciu, Alexander I., On Milnor fibrations of arrangements, J. London Math. Soc. (2) \textbf{51} (1995), 105-119.
\bibitem{DimcaPapadima}Dimca, Alexandru and Papadima, Stefan, Hypersurface complements, Milnor fibers and higher homotopy groups of arrangements, Annals of Mathematics, \textbf{158}(2003), 473-507.
\bibitem{DimcaPapadima2} Dimca, Alexandru and Papadima, Stefan, Equivariant chain complexes, twisted homology and relative minimality of arrangements, Ann. Sci. Ecole Norm. Sup. (4) \textbf{37} (2004), no. 3, 449-467.
\bibitem{Hatcher}Hatcher, Allen, Introduction to Algebraic Topology, Cambridge University Press, 2001.
\bibitem{Hattori}Hattori, Akio, Topology of $\mathbf{C^n}$ minus a finite number of affine hyperplanes in general position, J. Fac. Sci. Univ. Tokyo \textbf{22}, (1975), 205-219.
\bibitem{OrlikRandell}Orlik, Peter and Randell, Richard, The Milnor fiber of a generic arrangement, Arkiv f$\ddot{u}$r Mathematik 31 (1993), 71-81. 
\bibitem{OrlikTerao} Orlik, Peter and Terao, Hiroaki, Arrangements of Hyperplanes, Grund. Math. Wiss. \textbf{300}, Springer-Verlag Berlin Heidelberg, 1992.

\bibitem{RandellHomotopy}Randell, Richard, Homotopy and group cohomology of arrangements, Topology and its Applications \textbf{78} (1997), 201-213.

\bibitem{RandellMinimal}Randell, Richard, Morse theory, Milnor fibers and minimality of hyperplane arrangements, Proc. Amer. Math. Soc. \textbf{130} (2002), 2737-2743.

\bibitem{Rong} Rong, Liu, On the Classification of $(n-k+1)-$connected
Embeddings of $n-$Manifolds into $(n+k)$-manifolds in the metastable range,
Trans. Amer. Math. Soc., \textbf{347 },(1995), 4245-4258.
\bibitem{Yoshinaga} Yoshinaga, Masahiko, Generic section of a hyperplane arrangement and twisted Hurewicz maps, Topology Appl. \textbf{155} (2008), no. 9, 1022-1026.
\end{thebibliography}
\end{document}